\documentclass[11pt]{amsart}
\usepackage{amsfonts,amsmath,amsthm,amssymb,latexsym,mathrsfs}
\usepackage{tikz}
\usepackage{enumitem}

\usetikzlibrary{matrix,arrows,decorations.pathmorphing}
\usepackage{tikz-cd}
\tikzset{commutative diagrams/.cd}
\usepackage{theoremref}

\theoremstyle{theorem}
\newtheorem{theorem}{Theorem}[section]
\newtheorem{corollary}[theorem]{Corollary}
\newtheorem{lemma}[theorem]{Lemma}
\newtheorem{proposition}[theorem]{Proposition}

\theoremstyle{definition}

\newtheorem{example}[theorem]{Example}

\newtheorem{conjecture}[theorem]{Conjecture}

\theoremstyle{remark}
\newtheorem{remark}[theorem]{Remark}
\newtheorem*{warning}{Warning}
\numberwithin{equation}{section}
\def\Z{{\mathbb Z}}
\def\Q{{\mathbb Q}}
\def\C{{\mathbb C}}

\def\R{{\mathbb R}}

\def\Fp{{{\mathbb F}_p}}

\def\End{{\mathrm{End}}}

\def\subeq{\subseteq}

\def\bbar{\overline}

\def\incl{\hookrightarrow}

\def\bbH{{  \mathbb{H}  }}

\def\diag{{  \mathrm{diag}  }}

\begin{document}
\title{Unit equations on quaternions}
\author{Yifeng Huang}
\thanks{We thank Jason Bell, Dragos Ghioca, Tom Tucker and Michael Zieve for helpful conversations and useful comments. We thank the referee for the time devoted to reading the paper to its details, for valuable suggestions to the paper and \thref{conjecture}, and especially for generously sketching the proof of \thref{allcommutative} and \thref{virtualabelian} and informing the author about numerous useful results. This work was done with the support of Rackham One-term Dissertation Fellowship, Indu and Gopal Prasad Family Fund, and National Science Foundation grant DMS-1601844.}

\bibliographystyle{abbrv}

\begin{abstract}
A classical result about unit equations says that if $\Gamma_1$ and $\Gamma_2$ are finitely generated subgroups of $\C^\times$, then the equation $x+y=1$ has only finitely many solutions with $x\in\Gamma_1$ and $y\in \Gamma_2$. We study a noncommutative analogue of the result, where $\Gamma_1,\Gamma_2$ are finitely generated subsemigroups of the multiplicative group of a quaternion algebra. We prove an analogous conclusion when both semigroups are generated by algebraic quaternions with norms greater than $1$ and one of the semigroups is commutative. As an application in dynamics, we prove that if $f$ and $g$ are endomorphisms of a curve $C$ of genus $1$ over an algebraically closed field $k$, and $\deg(f), \deg(g)\geq 2$, then $f$ and $g$ have a common iterate if and only if some forward orbit of $f$ on $C(k)$ has infinite intersection with an orbit of $g$.
\end{abstract}
\maketitle

\section{Introduction}

A classical result about unit equations states that the equation $f+g=1$ has only finitely many solutions in a given finitely generated semigroup $\Gamma$ in $K^\times$, where $K$ is a field of characteristic zero. Unit equations have had important applications in many areas of mathematics, including Diophantine geometry (\cite{hindry2013diophantine,lang1960integral}), arithmetic dynamics \cite[p.\ 291]{evertse2015unit} and variants of the Mordell--Lang conjecture (for instance, see \cite[p.\ 321]{evertse2015unit}). Extensions of the classical result have also been studied, for example, see \cite{koymans2017equation,voloch1998} in the characteristic $p$ setting. 

 In this paper we present a class of semigroups in the standard quaternion algebra over $\R$ for which the finiteness of solutions of the unit equation holds. This is the first analogous result in the noncommutative setting. In light of the many applications of unit equations, this raises the intriguing possibility that some of those applications might have noncommutative analogues. 

Let $\bbH=\R\oplus \R i\oplus \R j\oplus \R k$ denote the quaternion algebra $\bbH$ over $\R$, with the standard multiplication law $i^2=j^2=k^2=-1, ij=-ji=k, jk=-kj=i, ki=-ik=j$. For an element $\alpha=a+bi+cj+dk\in \bbH$, where $a,b,c,d\in \R$, define its conjugation to be $\bbar\alpha=a-bi-cj-dk$, its norm to be $N(\alpha)=\alpha\bbar \alpha=\bbar\alpha \alpha=a^2+b^2+c^2+d^2$, and its trace $tr(\alpha)=\alpha+\bbar \alpha=2a$. Write $\lvert\alpha\rvert=\sqrt{N(\alpha)}$. 

We say that a quaternion $\alpha=a+bi+cj+dk\in \bbH$ is \textit{algebraic} if all coordinates $a,b,c,d$ are algebraic over $\Q$. This is equivalent to requiring that $\alpha$ satisfies a polynomial equation with coefficients in $\Q$, or that $\Q[\alpha]$ is a finite field extension of $\Q$. Indeed, $\alpha$ always satisfies the quadratic equation
\[X^2-tr(\alpha)X+N(\alpha)=0\]
and if $a,b,c,d\in \bbar\Q$, then so are $tr(\alpha)$ and $N(\alpha)$.

Denote by $\bbH_a$ the subalgebra of all quaternions that are algebraic.

\begin{theorem}\thlabel{main}
Let $\Gamma_1,\Gamma_2$ be semigroups of\/ $\bbH_a^\times$ generated by finitely many elements of norms greater than $1$, and fix $a,a',b,b'\in \bbH_a^\times$. If\/ $\Gamma_1$ is commutative, then the equation
\[
afa'+bgb'=1
\]
has only finitely many solutions with $f\in \Gamma_1$ and $g\in \Gamma_2$. In fact, for solutions $(f,g)$, we have effectively computable upper bounds for $\lvert f\rvert, \lvert g\rvert$ that depend only on $a,a',b,b'$ and generators of\/ $\Gamma_1,\Gamma_2$. 
\end{theorem}

We emphasize that even though $\Gamma_1$ is commutative, the semigroup $\Gamma_2$ need not be commutative, and that $a,a'$ and $\Gamma_1$ typically will not commute with each other. The proof relies on the following result, which implies that if a certain quaternion unit equation has infinitely many solutions, then so does another equation of a different type. We note that \thref{reduction} applies in greater generality than \thref{main}, as \thref{reduction} does not require $\Gamma_1$ to be commutative. 

\begin{theorem}\thlabel{reduction}
Let $\Gamma_1,\Gamma_2$ be semigroups of $\bbH_a^\times$ generated by finitely many elements of norms greater than $1$, and fix $a,a',b,b'\in \bbH_a^\times$. Then the equation
\[afa'+bgb'=1\]
has only finitely many solutions with $f\in \Gamma_1$ and $g\in \Gamma_2$ such that $\lvert 1-afa'\rvert\neq |afa'|$. In fact, for such pairs $(f,g)$, we have effectively computable upper bounds for $\lvert f\rvert, \lvert g\rvert$ that depend only on $a,a',b,b'$ and generators of\/ $\Gamma_1,\Gamma_2$.
\end{theorem}

Given \thref{reduction}, in order to prove \thref{main} it suffices to prove the next result which involves only the semigroup
$\Gamma_1$:

\begin{theorem}\thlabel{commutative}
Let $\Gamma$ be a semigroup generated by finitely many elements in $\bbH_a$ with norms greater than 1, and fix $a,a'\in \bbH_a^\times$. If\/ $\Gamma$ is commutative, then the equation
\[\lvert 1-afa'\rvert=\lvert afa'\rvert\]
has only finitely many solutions with $f\in \Gamma$. In fact, for solutions $f\in \Gamma$, we have an effectively computable upper bound for $\lvert f\rvert$ that depends only on $a,a'$ and generators of\/ $\Gamma$.
\end{theorem}

We remark that \thref{commutative} is the only step in the proof of \thref{main} that uses the commutativity of $\Gamma_1$, so any generalization of \thref{commutative} would immediately yield a generalization of \thref{main}. 

In light of the above results, we make the following conjecture about noncommutative unit equations:

\begin{conjecture}\thlabel{conjecture} Let $\Gamma_1,\Gamma_2$ be finitely generated semigroups of the multiplicative group $A^\times$ of a finite dimensional division algebra $A$ over $\Q$. Then for any fixed $a,a',b,b'\in A^\times$, the unit equation $afa'+bgb'=1$ has only finitely many solutions with $f\in \Gamma_1$ and $g\in \Gamma_2$.

Moreover, there is an effectively computable finite subset $S\subeq \Gamma_1\times\Gamma_2$ in terms of $a,a',b,b'$ and generators of\/ $\Gamma_1,\Gamma_2$, such that all solutions $(f,g)\in \Gamma_1\times\Gamma_2$ must lie in $S$.
\end{conjecture}

The referee kindly points out that the ineffective part of the conjecture is true in the case where all the semigroups are commutative:

\begin{proposition}\thlabel{allcommutative}
Let $k$ be a field of characteristic zero and $A$ be a finite-dimensional division algebra over $k$. Let $\Gamma_1,\dots,\Gamma_m$ be abelian and finitely generated subgroups of the multiplicative group $A^\times$. Then for any fixed $a_1,\dots,a_m,b_1,\dots,b_m\in A^\times$, the unit equation
\[a_1 f_1 b_1+\cdots+a_m f_m b_m=1\]
has only finitely many \emph{nondegenerate} solutions $(f_1,\dots,f_m)\in \Gamma_1\times\dots\times\Gamma_m$, i.e., solutions such that no proper subsum equals $1$.
\end{proposition}

This, of course, implies the case where all $\Gamma_i$ are \emph{virtually} abelian, in the sense that $\Gamma_i$ has a finite-index abelian subgroup. Moreover, as is mentioned by the referee, every semigroup that does not contain a free semigroup of rank two is contained in a virtually abelian subgroup of $A^\times$ (see \thref{virtualabelian}), so \thref{allcommutative} also holds if $\Gamma_i$ does not contain a free semigroup of rank two. We emphasize that \thref{main} is the only currently proven case of \thref{conjecture} where some of the semigroups are not contained in virtually abelian subgroups of $A^\times$. It is also the only known case where $A$ is noncommutative and one knows an effectively computable finite set that contains all the solutions.

In Section \ref{discussion}, we will discuss a possible $p$-adic approach to \thref{conjecture}, and will give a counterexample to the matrix algebra analogue of \thref{conjecture} in \thref{matrixexample}.

Our main theorem has the following consequence about intersections of orbits of endomorphisms of a genus-1 curve in arbitrary characteristic. 

\begin{corollary}\thlabel{dynamiccor}
Let $E$ be an elliptic curve over an algebraically closed field $k$, and let $f,g:E\to E$ be regular maps of degrees greater than 1. If there are points $A,B\in E(k)$ such that the forward orbits $O_f(A):=\{A,f(A),f^2(A),\ldots\}$ and $O_g(B):=\{B,g(B),g^2(B),\ldots\}$ have infinite intersection, then $f$ and $g$ have a common iterate, namely, $f^{m_0}=g^{n_0}$ for some positive integers $m_0,n_0$.

In fact, if $O_f(A)\cap O_g(B)$ is nonempty, let $m_0,n_0$ be integers such that $f^{m_0}(A)=g^{n_0}(B)$. Then there is an effectively computable constant $M$ in terms of $A, B, f, g, m_0, n_0$ such that, if $f^m(A)=g^n(B)$ for some $(m,n)$ where either $m>M$ or $n>M$, then $f^{m_0}=g^{n_0}$.

\end{corollary}

Analogous results have been proven in various cases in characteristic zero, in case $E$ is replaced by $\mathbb{A}^1$ \cite{ghioca2012linear}, a linear space \cite{ghioca2017orbit}, or a semiabelian variety \cite{ghioca2017orbit,ghioca2011mordell}. \thref{dynamiccor}, however, applies to all characteristics. 

It would be interesting to study high-dimensional analogues of \thref{dynamiccor}. For instance, we will show that if certain cases of \thref{conjecture} hold, then \thref{dynamiccor} remains true if $E$ is replaced by a \emph{simple} abelian variety, i.e., an abelian variety having no nonzero proper abelian subvarieties. The referee's \thref{allcommutative} thus yields an unconditional proof of the ineffective part of the simple abelian variety analogue of \thref{dynamiccor}.

\begin{corollary}\thlabel{dynamicAV}
Let $X$ be a simple abelian variety over an algebraically closed field $k$, and let $f,g:X\to X$ be regular maps of degrees greater than 1. If there are points $A,B\in X(k)$ such that the forward orbits $O_f(A):=\{A,f(A),f^2(A),\ldots\}$ and $O_g(B):=\{B,g(B),g^2(B),\ldots\}$ have infinite intersection, then $f$ and $g$ have a common iterate, namely, $f^{m_0}=g^{n_0}$ for some positive integers $m_0,n_0$.
\end{corollary}

The characteristic zero case of \thref{dynamiccor} and \thref{dynamicAV} is an instance of the higher-rank generalization posed in \cite[Question 1.6]{ghioca2012linear} of the dynamical Mordell--Lang conjecture \cite[Chapter 3]{bell2016dynamical}; see also \cite{ghioca2017orbit}. For positive characteristic, see \cite[Chapter 13]{bell2016dynamical}. We note that the conclusions of all previous results in characteristic $p>0$ involve the more complicated possibility of $p$-automatic sequences (e.g., \cite{derksen2007skolem,ghioca2019positive}), whereas the conclusion of \thref{dynamiccor} and \thref{dynamicAV} is more rigid. This extra possibility also occurs in the positive characteristic version of the original (not dynamical) Mordell--Lang conjecture \cite{moosascanlon2004}, where it is called an ``$F$-structure'' and where examples are given to show that the possibility cannot be removed.

The rest of the paper is organized as follows. In Section \ref{bakers}, we state a known Diophantine result. Then Section \ref{dynamicproof}, \ref{reductionproof} and \ref{commutativeproof} contain proofs of \thref{dynamiccor} (together with \thref{dynamicAV}), \thref{reduction} and \thref{commutative}, respectively. The proofs are independent of one another, and can be read in any order. \thref{main} follows immediately from \thref{reduction} and \thref{commutative}. The appendix includes the referee's proof of \thref{allcommutative} and a result about semigroups not containing noncommutative free semigroups.

\section{Linear Forms in Logarithms}\label{bakers}

The proofs of \thref{reduction} and \thref{commutative} rely on the following form of Baker's theorem on Diophantine approximation of logarithms. 

\begin{theorem}[Baker, W\"ustholz \cite{baker1993logarithmic}]\thlabel{baker}
Let $\lambda_1,\dots,\lambda_r$ be complex numbers such that $e^{\lambda_i}$ are algebraic for $1\leq i\leq r$. Then there are effectively computable constants $k,C>0$ depending on $r$ and $\lambda_i$ such that

\[0<\lvert a_1\lambda_1+\cdots+a_r\lambda_r\rvert\leq kH^{-C}\]
has no solutions in $a_i\in \Z$, where $H=\max_{i=1}^r \lvert a_i\rvert$.
\end{theorem}

The effective computability of \thref{baker} implies the effective part of our results, and our proofs will yield explicit bounds in our result if we use an explicit version of \thref{baker} (for example, see \cite[\S 3.2]{evertse2015unit}). 

\section{Proof of \thref{dynamiccor} and \thref{dynamicAV}}\label{dynamicproof}

In this section, we prove \thref{dynamiccor} and a conditional generalization to simple abelian varieties, which implies \thref{dynamicAV}. 

\begin{proof}[Proof of \thref{dynamiccor}]

Since $\deg(f)>1$, the regular map $f$ has a fixed point. By replacing the origin of $E$ by a fixed point of $f$ if necessary, we may assume that $f$ is an endomorphism of $E$. 

 Write  $g = \tau_Q \circ h$  where  $Q$  is a point on $E$,  $\tau_Q$  is the map  $E\to E$  defined by translation by $Q$,  and  $h$  is an endomorphism of  $E$.  Here  $\deg(h) = \deg(g) > 1$,  so that  $h-1$  is nonconstant  and thus induces a surjective map  $E \to E$.  Let  $R$  be a point on  $E$  such that  $(h-1)(R) = Q$.  Then, for any positive integer $n$, we have
\begin{equation}\label{ss0}g^n = \tau_{Q+h(Q)+h^2(Q)+\dots+h^{n-1}(Q)} \circ h^n = \tau_{ (h^n-1)(R) } \circ h^n.
\end{equation}

Thus, for any positive integer  $m$,  the condition  $f^m = g^n$  is equivalent to the conditions that $f^m = h^n$  and  $(h^n-1)(R) = O$.

Pick the orbits of $f$ and $g$ that have infinite intersection, and let  $P$  be any point in the intersection; then the orbits  $O_f(P)$  and  $O_g(P)$  also have infinite intersection, so there are infinitely many pairs $(m,n)$ of positive integers  such that
\[f^m(P) = g^n(P) = (h^n-1)(R) + h^n(P).\]

Fix such a pair $(m_0,n_0)$,  and  let $(m,n)$  be any other pair of positive integers that satisfy the above.  Then
\begin{equation}\label{ss1}
 (f^{m_0} - h^{n_0}) (P) = (h^{n_0}-1) (R)
\end{equation}   
\begin{equation}\label{ss2}
(f^m - h^n) (P) = (h^n-1) (R)
\end{equation}

Left-multiplying (\ref{ss1}) by the dual isogeny $(\bbar h^{n_0}-1)$ of $(h^{n_0}-1)$, we get
\[ (\bbar h^{n_0}-1)(f^{m_0}-h^{n_0})(P)=\deg(h^{n_0}-1)(R)\]

Left-multiplying further by $(h^n-1)$, we get
\[(h^n-1)(\bbar h^{n_0}-1)(f^{m_0}-h^{n_0})(P)=(h^n-1)\deg(h^{n_0}-1)(R)\]

Note that $\deg(h^{n_0}-1)$ is an integer, so it is in the center of $\End(E)$. Using (\ref{ss2}), we get
\[\left((h^n-1)(\bbar h^{n_0}-1)(f^{m_0}-h^{n_0})-(f^m - h^n)\deg(h^{n_0}-1) \right)(P)=O.\]

Since $O_f(P)$ is infinite, $P$ must be a point of infinite order (otherwise, $rP=0$ for some integer $r>0$, so $O_f(P)$ lies in the finite group  $E[r]$ of $r$-torion elements). 

Hence the kernel of $(h^n-1)(\bbar h^{n_0}-1)(f^{m_0}-h^{n_0})-(f^m - h^n)\deg(h^{n_0}-1)$ contains all (infinitely many) multiples of $P$. Since the kernel of any nonzero endomorphism is a finite group, we must have

\begin{equation}\label{ss3}
(h^n-1)(\bbar h^{n_0}-1)(f^{m_0}-h^{n_0})-(f^m - h^n)\deg(h^{n_0}-1)=0
\end{equation}
and recall that this holds for infinitely many pairs $(m,n)$.

Rewrite (\ref{ss3}) as an equation in $f^m$ and $h^n$:

\begin{equation}\label{ss4}
h^n(u+d)-f^m d=u
\end{equation} 
where $u=(\bbar h^{n_0}-1)(f^{m_0}-h^{n_0}), d=\deg(h^{n_0}-1)$.

Now $\End(E)\otimes_\Z \Q$ is either $\Q$ or an imaginary quadratic field or a positive definite quaternion algebra over $\Q$, all of which can be embedded into some positive definite quaternion algebra $H$ over $\Q$. View the equation (\ref{ss4}) in $H$. 

If $u\neq 0$, then the equation $h^n(u+d)u^{-1}-f^m du^{-1}=1$ has infinitely many solutions $m,n>0$, a contradiction to \thref{main} with $a=b=1$, $a'=(u+d)u^{-1}$, $b'=-du^{-1}$, $\Gamma_1$ generated by $h$, and $\Gamma_2$ generated by $f$. Hence $u=0$, so that $(\bbar h^{n_0}-1)(f^{m_0}-h^{n_0})=0$.

But $\deg \bbar h=\deg h>1$ implies $\bbar h^{n_0}-1\neq 0$, so $f^{m_0}=h^{n_0}$.

Finally, equation (\ref{ss1}) implies $(h^{n_0}-1)(R)=O$, so $g^{n_0}=h^{n_0}=f^{m_0}$ by (\ref{ss0}).

\end{proof}

Let $X$ be a simple abelian variety, and assume that \thref{conjecture} holds for $A=\End(X)\otimes \Q$ and $\Gamma_1,\Gamma_2$ being cyclic semigroups. We claim that \thref{dynamiccor} remains true if $E$ is replaced by $X$. The proof is the same is above except for three places. First, we used the fact that $h-1$ is surjective because it is nonconstant. This is still true because the image of a morphism must be an abelian subvariety, but $X$ is simple. Second, we used the elements $\bbar h^{n_0}-1$, but dual isogeny no longer exists in abelian varieties in general. However, we can fix an endomorphism $\varphi$ such that $\varphi \circ (h^{n_0}-1)=\deg(h^{n_0}-1)$, and use $\varphi$ in place of $\bbar h^{n_0}-1$. Third, we used the argument that if an endomorphism $\psi$ of $E$ vanishes at a point $P$ of infinite order, then $\psi=0$. This is also true for simple abelian variety $X$: the endomorphism $\psi$ must vanish on the Zariski closure of the group genearated by $P$, but it contains an abelian subvariety of $X$ of positive dimension, which has to be the whole $X$ because $X$ is simple. 

Given the referee's \thref{allcommutative}, the ineffective part of the conditional result above holds unconditionally, and it gives an unconditional proof of \thref{dynamicAV}.

\begin{remark}
When $k$ has characteristic zero, the ineffective part of \thref{dynamiccor} was proved via two different methods in \cite[Theorem 1.4]{ghioca2017orbit} and \cite[Theorem 1.2.3]{odesky}. Our proof of \thref{dynamiccor} extends the latter proof to arbitrary characteristic, and in fact the possibility of such an extension was the initial motivation for studying unit equations on quaternions in the present paper. We thank Michael Zieve for informing the author about \cite[Theorem 1.2.3]{odesky} and suggesting this possibility.
\end{remark}

\begin{remark}
If $f,g$ are endomorphisms of an elliptic curve $E$ without translation, then \thref{dynamiccor} becomes trivial. For a proof, set $P\in E(k)$ be a point in the intersection of orbits, and let $n,m>0$ be such that $f^n(P)=g^m(P)$. For any integer $N$, we have $Nf^n(P)=Ng^m(P)$, so that $(f^n-g^m)(NP)=O$ because $f,g$ are endomorphisms of $E$. But $P$ is of infinite order (otherwise the forward orbit of $P$ under $f$ would be finite), so $\ker(f^n-g^m)$ is an infinite group, and the only possibility is $f^n-g^m=0$.
\end{remark}

\section{Proof of \thref{reduction}}\label{reductionproof}

Let $\Delta$ be the set consisting of $(f,g)\in \Gamma_1\times \Gamma_2$ such that $afa'+bgb'=1$ and $\lvert 1-afa'\rvert\neq \lvert afa'\rvert$. Then the goal of \thref{reduction} is precisely to show that $\Delta$ is a finite set. 

By triangle inequality, every $(f,g)\in \Delta$ satisfies
\begin{equation}\label{eq1}
0<\Bigl\lvert \lvert afa'\rvert -\lvert bgb'\rvert \Bigr\rvert\leq 1
\end{equation}

We observe that since $\Gamma_i$ ($i=1,2$) is a semigroup generated by finitely many elements with norms greater than 1, there are only finitely many elements of $\Gamma_i$ of bounded norm.

In the rest of the proof, we will prove the claim that $\{\lvert f\rvert: (f,g)\in \Delta\}$ is bounded. Given the claim, the set $\{f:(f,g)\in\Delta\}$ is finite  by the observation above. Since $f$ determines $g$ by $g=b^{-1}(1-afa')b'^{-1}$, there are only finitely many choices for $g$ as well, and \thref{reduction} is proved. 

For contradiction, we assume that there is a solution $(f,g)\in\Delta$ with arbitrarily large $\lvert f\rvert$. Using simple calculus (specifically, Lagrange's mean value theorem), (\ref{eq1}) implies 
\begin{equation}\label{eq2}
0<\Bigl\lvert \log\lvert afa'\rvert -\log\lvert bgb'\rvert \Bigr\rvert \leq \frac{2}{\lvert afa'\rvert}
\end{equation}
for sufficiently large $\lvert f\rvert$. 

Let the semigroup $\log\lvert \Gamma_1\rvert$ be generated by $x_1,\dots,x_t>0$ and $\log\lvert \Gamma_2\rvert$ by $y_1,\dots,y_u>0$. Write $\log\lvert f \rvert=m_1x_1+\dots+m_tx_t$, $\log\lvert g\rvert=n_1y_1+\dots+n_ty_t$ for some nonnegative integers $m_i, n_j$. Let $c=\log\lvert aa'/bb'\rvert$. Then $c,x_i,y_j$ are logarithms of real algebraic numbers, and (\ref{eq2}) can be rewritten as
\begin{equation}\label{eq3}
0<\lvert c+m_1x_1+\dots+m_tx_t-n_1y_1-\dots-n_uy_u\rvert \leq \frac{2}{\lvert a\rvert e^{x_1m_1+\dots+x_tm_t}}
\end{equation}

By \thref{baker} (Baker's theorem), there are positive constants $k,C$ such that 
\[
0<\lvert a_1 c+m_1x_1+\dots+m_tx_t-n_1y_1-\dots-n_uy_u\rvert \leq k\max\{\lvert a_1\rvert,\lvert m_i\rvert,\lvert n_j\rvert\}^{-C}
\]
has no integer solution $(a_1,m_1,\dots,m_t,n_1,\dots,n_u)$. In particular, for $a_1=1$ and $m_i,n_j>0$, the inequality
\begin{equation}\label{eqbaker}
0<\lvert c+m_1x_1+\dots+m_tx_t-n_1y_1-\dots-n_uy_u\rvert \leq kH^{-C}\text{ has no solution, }
\end{equation}
where $H=\max\{1,m_1,\dots,m_t,n_1,\dots,n_u\}$.

Our next goal is to bound the right-hand side of (\ref{eq3}) by a function of $H$, in order to reach a contradiction with (\ref{eqbaker}). Since $x_i, y_j$ are positive, for $|f|$ sufficiently large and satisfying (\ref{eq3}), it is not hard to see that
\begin{equation}\label{eq4}
C_1\max\{m_i\}<\max\{n_j\}<C_2\max\{m_i\}
\end{equation}
for some $C_1, C_2>0$ that does not depend on $m_i, n_j$. For a proof, we note that
\[\min\{x_i\}\max\{m_i\}\leq m_1x_1+\dots+m_tx_t \leq t\max\{x_i\}\max\{m_i\}\]
\[\min\{y_j\}\max\{n_j\}\leq n_1y_1+\dots+n_uy_u \leq u\max\{y_j\}\max\{n_j\}\]
and (\ref{eq3}) gives
\[\frac{1}{2}(n_1y_1+\dots+n_uy_u)<m_1x_1+\dots+m_tx_t<2(n_1y_1+\dots+n_uy_u)\]
for sufficiently large $|f|$. Hence $\max\{m_i\},\max\{n_j\},\log|f|$ and $\log|g|$ are all ``comparable" to each other in the sense of (\ref{eq4})

It follows that
\begin{equation}\label{eq5}
C_1\max\{m_i\}<H\leq \max\{C_2,1\}\max\{m_i\}=:C_2'\max\{m_i\}
\end{equation}
where we denote $C_2'=\max\{C_2,1\}$.

Now (\ref{eq3}) implies 
\begin{equation}\label{eq6}
0<\lvert c+m_1x_1+\dots+m_tx_t-n_1y_1-\dots-n_uy_u\rvert \leq \frac{2}{\lvert a\rvert e^{\min\{x_i\}\max{m_i}}}\leq\frac{2}{\lvert a\rvert e^{\min\{x_i\}H/C_2'}}
\end{equation}
for sufficiently large $\lvert f\rvert$ (or equivalently, $H$, by the ``comparability" discussion above together with (\ref{eq5})).

Since the right-hand side decays exponentially in $H$, it will be less than $kH^{-C}$ for large $H$, which contradicts the lack of solution of (\ref{eqbaker}). \qed

\section{Proof of \thref{commutative}}\label{commutativeproof}

First, we observe that the equation $\lvert 1-afa'\rvert=\lvert afa'\rvert$ can be rewritten as $\lvert a^{-1}a'^{-1}-f\rvert=\lvert 0-f\rvert$. Note that $|\cdot|$ is the norm induced from the inner product on $\bbH$ with $\{1,i,j,k\}$ being an orthonormal basis. We denote the inner product by $\langle\cdot,\cdot\rangle$.  

Denoting $d=a^{-1}a'^{-1}$, the equation above gives
\begin{align*}
\langle f,f\rangle &= \lvert f\rvert^2\\
&=\lvert d-f \rvert^2\\
&=\langle d-f, d-f\rangle = \lvert d\rvert^2 - 2\langle d,f\rangle + \lvert f\rvert^2,
\end{align*}
which simplifies to $2\langle d,f\rangle = \lvert d\rvert^2$.

Hence the equation is equivalent to that $f$ lies in a hyperplane not passing through the origin, given by
\[\{x\in \bbH: \langle a^{-1}a'^{-1}, x\rangle=\frac{1}{2}\lvert a^{-1}a'^{-1}\rvert^2\}.\]

Given the observation above, \thref{commutative} follows from the following lemma:

\begin{lemma}
Let $\Gamma$ be a commutative semigroup of\/ $\bbH^\times$ generated by finitely many algebraic elements of norms greater than 1, and $H$ be a hyperplane of\/ $\bbH$ defined by
\[H=\{x\in \bbH: \Theta(x)=1\}\]
where $\Theta: \bbH\to \R$ is a nonzero $\R$-linear functional that maps $\bbH_a$ into $\bbar \Q\cap \R$. Then $\Gamma\cap H$ is finite. In fact, we have an effectively computable upper bound that depends only on $H$ and $\Gamma$ for norms of elements of\/ $\Gamma\cap H$. 
\end{lemma}

\begin{proof}[Proof of lemma]
Since $\Gamma$ is commutative, it lies in a subalgebra in $\bbH$ that is isomorphic to $\C$. Passing to its restriction on this subalgebra, we may assume instead that $\Gamma$ is a semigroup generated by $g_1,\dots,g_s\in \bbar \Q^\times\subeq \C$ such that $\lvert g_j\rvert>1$, and $\Theta:\C\to \R$ is an $\R$-linear functional (which could now be zero) that maps $\bbar \Q$ into $\bbar \Q\cap \R$. We need to show that $\Theta(f)=1$ has only finitely many solutions $f\in \Gamma$. 

There is no question to ask if $\Theta=0$. In the case $\Theta\neq 0$, we may assume $\Theta$ is given by $\langle v,\cdot\rangle$ for some nonzero vector $v\in \bbar \Q$, where $\langle \cdot,\cdot\rangle$ is the standard Euclidean inner product on $\C$ with $\{1,i\}$ being an orthonormal basis. By rescaling, we may assume $\lvert v\rvert=1$, but the equation $\Theta(f)=1$ will become
\begin{equation}\label{eq7}
\langle v,f\rangle=M
\end{equation}
for some real algebraic number $M>0$.

Write $g_j=r_j v_j$ with $r_j>1$ and $v_j=e^{i\theta_j}$ on the unit circle, with $0\leq \theta_j<2\pi$. Also write $v=e^{i\theta}$ with $0\leq \theta<2\pi$. For $f=g_1^{n_1}\dots g_s^{n_s}$, the equation (\ref{eq7}) becomes
\begin{equation}\label{eq8}
\langle v,e^{i(n_1\theta_1+\dots +n_s\theta_s)}\rangle = Mr_1^{-n_1}\dots r_s^{-n_s}
\end{equation}

The left-hand side involves the inner product of two unit vectors, so its value is $\cos((n_1\theta_1+\dots +n_s\theta_s)-\theta)$. When $n_i$ are sufficiently large, the right-hand side of \ref{eq8} is small. But $\lvert\cos((n_1\theta_1+\dots +n_s\theta_s)-\theta)\rvert$ is approximately the closest distance from $(n_1\theta_1+\dots +n_s\theta_s)-\theta$ to $(m+1/2)\pi$ for integer $m$. If (\ref{eq8}) is satisfied by infinitely many $(n_j)$'s, then for sufficiently large solutions $(n_j)$, we have
\begin{equation}\label{eq9}
0<\left\lvert\left(\frac{\pi}{2}+\theta\right)+m\pi-(n_1\theta_1+\dots +n_s\theta_s)\right\rvert<2Mr_1^{-n_1}\dots r_s^{-n_s}
\end{equation}
for some $m\in \Z$.

By assumption, $v, v_j$ are algebraic numbers, so $\lambda:=i(\frac{1}{2}\pi+\theta)$, $\mu=i\pi$ and $\lambda_j=i\theta_j$ are logarithms of algebraic numbers. By \thref{baker}, there are constants $k,C>0$ such that the inequality 
\begin{equation}\label{eq10}
0<\left\lvert\left(\frac{\pi}{2}+\theta\right)+m\pi-(n_1\theta_1+\dots +n_s\theta_s)\right\rvert<kB^{-C} \text{ has no solution}
\end{equation}
for $m, n_j\in \Z, n_j\geq 0$, where
\[B=\max\{1,|m|,n_j\}\]

But for solutions of (\ref{eq9}) with $n_j$ large, $m\pi$ must be close to $n_1\theta_1+\dots +n_s\theta_s-(\frac{1}{2}\pi+\theta)$. Noting that
\[n_1\theta_1+\dots +n_s\theta_s\leq s\max\{\theta_j\}\max\{n_j\},\]
we have
\begin{equation}\label{eq11}
\lvert m\rvert\leq C'\max\{n_j\}
\end{equation}
for some constant $C'$, and thus
\begin{equation}\label{eq12}
\max\{n_j\}\leq B=\max\{n_j,\lvert m\rvert\}\leq \max\{1,C'\}\max\{n_j\}
\end{equation}

It follows from (\ref{eq10}) that for some constant $k'>0$, 
\begin{equation}\label{eq13}
0<\left\lvert\left(\frac{\pi}{2}+\theta\right)+m\pi-(n_1\theta_1+\dots +n_s\theta_s)\right\rvert<k'\max\{n_j\}^{-C} \text{ has no solution}
\end{equation}
for $m,n_j\in \Z, n_j\geq 0$. But for $(n_j)$ large, $2Mr_1^{-n_1}\dots r_s^{-n_s}<k'\max\{n_j\}^{-C}$, yielding a contradiction with (\ref{eq9}). 
\end{proof}

\section{Future Work}\label{discussion}

We were able to arrive at the main theorem using the archimedean norm only.  If we can furthermore use some version of $p$-adic norm on the division algebra $A$, we can vastly improve the result by applying K. Yu's theorem about $p$-adic logarithms in \cite{yu2007p}. One possible proposal for a $p$-adic norm is to use the reduced norm of a division algebra over $\Q_p$, which only works if $A\otimes \Q_p$ is still a division algebra. Unfortunately, for each given $A$, this only holds for finitely many $p$. 

\thref{reduction} is potentially useful for more cases than in \thref{main}. For example, one can explore the analogue of \thref{commutative} in the case where $\Gamma$ has two or more noncommutative generators, and then apply \thref{reduction}. Even if $\Gamma$ is replaced by its subset $\{f_1^{n_1}f_2^{n_2}:n_1,n_2\geq 0\}$, where $f_1,f_2$ are noncommutative generators with norms greater than 1, the analogue of \thref{commutative} remains open. 

The following example shows that we should only consider \thref{conjecture} where $A$ is a division algebra. 

\begin{example}\thlabel{matrixexample}
Take $A=M_2(\Q)$, the algebra of $2\times 2$ matrices over $\Q$. Then the multiplicative semigroup generated by $\begin{bmatrix} 1 & 1\\ 0 & 1 \end{bmatrix}$ is 
\[\Gamma:=\left\{\begin{bmatrix}
 1&n\\0&1
\end{bmatrix}: n\in \Z, n\geq 0 \right\}.\] 
The equation $2f-g=1_A$ has infinitely many solutions $f,g\in \Gamma$, namely all $(f,g)$ with $f\in\Gamma$ and $g=f^2$. 
\end{example}

\section*{Appendix}
\renewcommand{\thesection}{\Alph{section}}
\stepcounter{section}
\setcounter{section}{1}
This section contains the proofs of \thref{allcommutative} and \thref{virtualabelian}, both sketched by the referee. We start with an observation that will be used in both proofs.

\begin{lemma}\thlabel{embed}
Let $A$ be a finite-dimensional division algebra over a field $k$, and let $K$ be the algebraic closure of $k$. Then there is an embedding of $k$-algebras from $A$ to the matrix algebra $M_n(K)$ for some integer $n>0$.
\end{lemma}
\begin{proof}
Let $L$ be the center of $A$. Then $L$ is a finite extension of $k$ and we can embed $A$ into $A\otimes_L K$, which is isomorphic to the matrix algebra $M_n(K)$ for some integer $n$ by a standard fact about central simple algebras.
\end{proof}

\begin{proof}[Proof of \thref{allcommutative}]
Let $K$ be the algebraic closure of $k$ and fix an embedding $A\incl M_n(K)$ as in \thref{embed}. Note that nonzero elements of $A$ are sent to invertible matrices in $M_n(K)$. From now on, we shall consider the unit equation in $M_n(K)$. 

To set up a proof by contradiction, we assume that
\begin{equation}\label{shortest}
a_1f_1b_1+\dots+a_mf_mb_m=1, (f_1,\dots,f_m)\in \Gamma_1\times\dots \times \Gamma_m
\end{equation}
is a shortest equation (i.e., with minimal $m$) in the setting of \thref{allcommutative} that has infinitely many degenerate solutions. We claim:

\begin{equation}\label{claim1}
\begin{aligned}
&\text{There cannot exist an infinite family of degenerate solutions }\\ &\{(f_1^\alpha,\dots,f_m^\alpha)\} \text{ indexed by }\alpha\text{ in an infinite set, such that }f_1^\alpha\\&\text{are the same for all }\alpha.
\end{aligned}
\end{equation}

Otherwise, call $f_1^\alpha=f_1$, and let $u=a_1f_1b_1$, which is not $1$ because the solution is nondegenerate. Then $1-u$ is a unit in $A$ because $A$ is a division algebra, and set $b'_i=b_i (1-u)^{-1}$. The following equation
\[a_2f_2b'_2+\dots+a_mf_mb'_m=1\]
has infinitely many nondegenerate solutions $(f_2,\dots,f_m)=(f_2^\alpha,\dots,f_m^\alpha)$, contradicting the minimality of $m$. 

Now note that every element $\gamma$ of $\Gamma_i$ is diagonalizable in $M_n(K)$. Indeed, since $\gamma\in R$ and $R$ is finite-dimensional over $k$, we see that $\gamma$ satisfies some minimal polynomial $p(\gamma)=0$ where $p(x)\in k[x]$ is monic. To show that $\gamma$ is diagonalizable, it suffices to show that $p(x)$ has no repeated root in $K$. Assume the contrary, then there is a proper divisor $p_0(x)\in k[x]$ of $p(x)$ such that $p(x)$ divides $p_0(x)^2$. Thus $p_0(\gamma)^2=0$, so that $p_0(\gamma)=0$ because $R$ is a division algebra. This is a contradiction to the minimality of $p(x)$.

Since $\Gamma_i$ is abelian and finitely generated, and every element of $\Gamma_i$ is diagonalizable, there is a simultaneous diagonalization of $\Gamma_i$ by some $s_i\in GL_n(K)$, i.e., $s_i\Gamma_i s_i^{-1}$ only consists of diagonal matrices in $M_n(K)$. So we may replace $a_i$ by $a_i s_i^{-1}$, $b_i$ by $s_i b_i$, and $\Gamma_i$ by $s_i\Gamma_i s_i^{-1}$ and assume that each $\Gamma_i$ only consists of diagonal matrices in $M_n(K)$ (though $\Gamma_i, a_i, b_i$ are no longer inside $R$, but it will not matter). 

Now consider a solution $(f_i)$ of $1=a_1 f_1 b_1+\dots+a_m f_m b_m$, where $f_i=\diag(x_{i1},\dots,x_{in})$. Looking at the $(1,1)$-entry, we obtain an equation of the form 
\begin{equation}\label{entrywise}
1=\sum_{1\leq i,j\leq n} p_{ij} x_{ij}
\end{equation}
for some fixed $p_{ij}\in K$. (To see it, one merely needs to notice that every entry of $a_1 f_1 b_1+\dots+a_m f_m b_m$ is a linear combination of entries of $f_i$. )

Let $S$ be the set consisting of all $(i,j)$ such that $p_{ij}\neq 0$. We may assume that $S$ is nonempty, otherwise \eqref{entrywise} has no solution and there is nothing to prove. We claim that
\begin{equation}\label{claim2}
\begin{aligned}
&\text{There are finite sets $X_{ij}$ for $(i,j)\in S$, such that whenever $(x_{ij})$ is} \\
&\text{a solution of \eqref{entrywise}, there exists $(i_0,j_0)\in S$ such that $x_{i_0 j_0}\in X_{i_0 j_0}.$}\\
\end{aligned}
\end{equation}

To prove the claim, notice that there is a finitely generated subgroup of $K^\times$ that contains all $x_{ij}$ because $\Gamma_i$ is finitely generated. By the $S$-unit theorem in several variables \cite[Theorem 6.1.3]{evertse2015unit}, for every nonempty subset $T$ of $S$, the equation
\begin{equation}\label{subsum}
1=\sum_{(i,j)\in T} p_{ij} x_{ij}
\end{equation}
has only finitely many nondegenerate solutions. Let $X_{ij}$ be the set consisting of all $x_{ij}$ that appears in a nondegenerate solution of \eqref{subsum} for some $T$. Now for every solution $(x_{ij})$ of \eqref{entrywise}, there is a minimal nonempty subset $T_0$ of $S$ such that $1=\sum_{(i,j)\in T_0} p_{ij} x_{ij}$. Pick $(i_0,j_0)\in T_0$, then by construction, $x_{i_0 j_0}\in X_{i_0 j_0}$, as required.

Now applying the pigeonhole principle to the infinitely many solutions of \eqref{shortest}, we see that there exists $(i_0,j_0)\in S$ and $x\in X_{i_0 j_0}$ such that there are infinitely many solutions $(f_i)=(\diag(x_{ij}))$ with $x_{i_0 j_0}=x$. Without loss of generality, assume $(i_0,j_0)=(1,1)$. We claim that all those solutions $(f_i)$ have the same $f_1$. This contradicts \eqref{claim1} and finishes the proof of \thref{allcommutative}.

It remains to prove the claim. Let $(f_i)=\diag(x_{i1},\dots,x_{in})$ and $(f'_i)=\diag(x'_{i1},\dots,x'_{in})$ be two solutions of \eqref{shortest} such that $x_{11}=x'_{11}=x$. Then $g:=f_1 (f'_1)^{-1}$ is a diagonal matrix with $(1,1)$-entry being $1$. In particular, $g-1$ is not invertible in $M_n(K)$. But by construction, $h:=s_1^{-1}(g-1)s_1$ is in $A$, so $h$ is either zero or a unit. Since $g-1$ is not invertible, neither is $h$. Therefore $h=0$, so that $g=1$, which gives $f_1=f'_1$. 
\end{proof}

We next prove a statement that implies a generalization of \thref{allcommutative} where all $\Gamma_i$ are semigroups not containing free semigroups of rank two. Recall that a group $G$ is \emph{virtually} $P$ (where $P$ is a property) if $G$ has a finite-index subgroup that is $P$. 

\begin{proposition}\thlabel{virtualabelian}
Let $A$ be a division algebra over a field $k$, and $\Gamma$ be a finitely generated semigroup of $A^\times$. If $\Gamma$ does not contain a free semigroup of rank two, then the group $G$ generated by $\Gamma$ in $A^\times$ is virtually abelian. 
\end{proposition}
\begin{proof}
First, we will show that $G$ is virtually nilpotent using Theorem 1 of \cite{okninskisalwa1995}. Embed $A$ into $M_n(K)$ as in \thref{embed}, where $K$ is the algebraic closure of $k$. Note that $\Gamma$ is in $GL_n(K)$. 

Since $\Gamma$ is finitely generated, there is a finitely generated subfield $L\subeq K$ such that $\Gamma\subeq GL_n(L)$ (for example, let $L$ be the field generated by matrix entries of generators of $\Gamma$ over the prime field $\Q$ or $\Fp$ of $k$). Now, Theorem 1 of \cite{okninskisalwa1995} implies that $G$ is virtually nilpotent.

Let $N$ be a finite-index nilpotent subgroup of $G$. By the following lemma, $N$ is virtually abelian, and the proof is complete.  
\end{proof}

\begin{warning}
Here $L$ may not contain $k$, but it does not matter for the purpose of this proof.
\end{warning}

\begin{lemma}
If $A$ is a division algebra over a field $k$, and $N\subeq A^\times$ is a solvable subgroup, then $N$ is virtually abelian.
\end{lemma}
\begin{proof}
Again, we embed $A$ into $M_n(K)$ as in \thref{embed}, where $K$ is the algebraic closure of $k$. Then $N$ is a subgroup of $GL_n(K)$. The Zariski closure $\bbar N$ of $N$ in $GL_n(K)$ is a $K$-algebraic group that is still solvable, and so is its identity component $\bbar N_0$. Let $N_0:=\bbar N_0\cap N$. Since $\bbar N_0$ has finite index in $\bbar N$, so does $N_0$ in $N$. 

We claim that $N_0$ is abelian. By the Lie--Kolchin triangularization theorem \cite[Theorem 6.3.1]{springer1998linear}, there is $s\in GL_n(K)$ such that $s\bbar N_0 s^{-1}$ consists of upper triangular matrices, so $sN_0s^{-1}$ does as well. We observe that if $a$ and $b$ are invertible upper triangular matrices, then the commutator $[a,b]:=aba^{-1}b^{-1}$ is in $U$, the group of upper triangular matrices with diagonal entries all $1$. It follows that $s[N_0,N_0]s^{-1}\subeq U$.

It remains to show that $[N_0,N_0]$ is the trivial group. Take $x\in [N_0,N_0]$, and note that $x\in A$, so $x-1\in A$ is either zero or invertible. But $s(x-1)s^{-1}=sxs^{-1}-1$ is an upper triangular matrix in $Mat_n(K)$ with diagonal entries all $0$, so $x-1$ cannot be invertible. It follows that $x=1$ and $N_0$ is abelian.
\end{proof}


\end{document}